\documentclass[11pt]{amsart}
\usepackage{amsmath,amsthm,amsfonts,amscd,amssymb,eucal,latexsym,mathrsfs,stmaryrd,enumitem,mathtools}
\usepackage[all]{xy}

\newtheorem{theorem}{Theorem}[section]

\newtheorem{lemma}[theorem]{Lemma}

\theoremstyle{definition}

\newtheorem{remark}[theorem]{Remark}

\newtheorem{question}[theorem]{Question}

\newcounter{theoremintro}
\newtheorem{theoremi}[theoremintro]{Theorem}

\newcommand{\id}{{\rm id}}

\newcommand{\sC}{{\mathscr C}}
\newcommand{\sD}{{\mathscr D}}

\newcommand{\sL}{{\mathscr L}}

\newcommand{\sP}{{\mathscr P}}
\newcommand{\sQ}{{\mathscr Q}}

\newcommand{\sS}{{\mathscr S}}

\newcommand{\Zb}{{\mathbb Z}}

\newcommand{\Nb}{{\mathbb N}}

\newcommand{\eps}{\varepsilon}

\DeclareMathOperator{\dom}{dom}

\allowdisplaybreaks

\begin{document}

\title{Entropy, virtual Abelianness, and Shannon orbit equivalence}
\author{David Kerr}
\address{David Kerr,
Mathematisches Institut,
WWU M{\"u}nster,
Einsteinstr.\ 62,
48149 M{\"u}nster, Germany}
\email{kerrd@uni-muenster.de}

\author{Hanfeng Li}
\address{Hanfeng Li,
Department of Mathematics, SUNY at Buffalo, Buffalo, NY 14260-2900, USA
}
\email{hfli@math.buffalo.edu}

\date{February 20, 2022}

\begin{abstract}
We prove that if two free p.m.p.\ $\Zb$-actions are Shannon orbit equivalent then they
have the same entropy. The argument also applies more generally to yield the same conclusion
for free p.m.p.\ actions
of finitely generated virtually Abelian groups.
Together with the isomorphism theorems of Ornstein and Ornstein--Weiss
and the entropy invariance results of Austin and Kerr--Li in the non-virtually-cyclic setting,
this shows that two Bernoulli actions of any non-locally-finite countably infinite amenable group
are Shannon orbit equivalent if and only if they are measure conjugate.
We also show, at the opposite end of the stochastic spectrum,
that every $\Zb$-odometer is Shannon orbit equivalent to the universal $\Zb$-odometer.
\end{abstract}

\maketitle

On the surface it would seem that there is a fundamental incompatibility between the concepts of entropy
and orbit equivalence.
On the one hand Ornstein and Weiss proved, building on a seminal theorem of Dye in the integer case,
that any two free ergodic p.m.p.\ (probability-measure-preserving)
actions of a countably infinite amenable group are orbit equivalent,
so that any asymptotic information about the dynamics or the group that is not a mere corollary
to ergodicity or amenability can be completely obliterated \cite{Dye59,OrnWei80,OrnWei87}.
On the other hand it is precisely for amenable groups that the original Kolmogorov--Sinai form
of entropy based on the asymptotic averaging of Shannon entropies
is most naturally and generally defined,
and in that setting a rich theory has developed whose highlights include the classification
of Bernoulli actions up to measure conjugacy by their entropy, first established by
Ornstein for the integers and then by Ornstein and Weiss for all countably infinite amenable groups \cite{Orn70,OrnWei80,OrnWei87}.

It turns out, however, that there is a certain
combinatorial robustness with which entropy registers the higher-order statistics of set intersections,
permitting a degree of reshuffling in the way that orbits are parameterized.
Indeed this became manifest in the Ornstein theory of Bernoulli actions and in related research
on Kakutani equivalence \cite{Kat77,OrnRudWei82,delJunRud84}, as well as in the work of
Vershik on actions of locally finite groups \cite{Ver73,Ver95}. Ultimately these threads were
united in the abstract theory of Kammeyer and Rudolph
which describes, for actions of a fixed countable amenable group, how entropy may or may not be preserved
under orbit equivalences subject to various conditions, and how entropy can serve as a
complete invariant, up to some restricted notion of orbit equivalence,
within certain classes of actions \cite{KamRud97,KamRud02,Rud05}.

More recently, taking a rather different geometric perspective not
based on the combinatorial shuffling of orbit segments, and inspired by growing interest in
integrable notions of equivalence in the study of measure rigidity,
Austin proved that if $G$ and $H$ are any two finitely generated amenable groups and
$G\curvearrowright (X,\mu )$ and $H\curvearrowright (Y,\nu )$
are free p.m.p.\ actions which are integrably orbit equivalent (i.e., there is an orbit equivalence between them whose
cocycles, when restricted to each group element and composed with a fixed word-length metric, give integrable functions)
then the actions have the same entropy \cite{Aus16}.
Curiously, a separate argument was required to handle the virtually cyclic case
(one based in fact on Kakutani equivalence), while the proof
in the non-virtually-cyclic case actually shows that the equality of entropies still holds if one instead
assumes the weaker relation of Shannon orbit equivalence, in which the cocycle partitions
associated to the orbit equivalence are assumed to have finite Shannon entropy.
In \cite{KerLi21} this entropy invariance under Shannon orbit equivalence was verified to hold more generally
when $G$ and $H$ are any countable amenable groups which are neither virtually cyclic nor locally finite.
The non-local-finiteness assumption is necessary here,
as a theorem of Vershik shows that any two nontrivial Bernoulli actions of a countably infinite locally finite group
are boundedly orbit equivalent \cite{Ver73,Ver95} (see the discussion in the introduction to \cite{KerLi21}).
It has remained an open question however whether the non-virtual-cyclicity condition can be dropped,
and in particular whether a Shannon orbit equivalence between free p.m.p.\ $\Zb$-actions
preserves entropy.
The first goal of the present paper is to establish the following.

\begin{theoremi}\label{T-entropy}
Let $G$ and $H$ be finitely generated virtually Abelian groups.
Let $G\stackrel{\alpha}{\curvearrowright} (X,\mu )$
and $H\stackrel{\beta}{\curvearrowright} (Y,\nu )$ be free p.m.p.\ actions
which are Shannon orbit equivalent. Then $h(\alpha ) = h(\beta )$.
\end{theoremi}

The proofs of entropy invariance for Shannon orbit equivalence in \cite{Aus16,KerLi21}
rely on the existence of sparse but coarsely dense trees inside of F{\o}lner sets, which
requires the group to have superlinear growth and thereby rules out virtually cyclicity.
The entropy already registers along these trees when it is computed at a fine enough resolution,
while the exponential complexity of the cocycles, relative to the size
of the ambient F{\o}lner set, is small when restricted to the tree, so that one can effectively
assume the cocycle values along the tree to be constant and thus transfer the
entropy growth over to the second action. This already yields Theorem~\ref{T-entropy}
in the case that neither $G$ nor $H$ is virtually cyclic, and also yields an entropy inequality
if one of $G$ and $H$ is not virtually cyclic. For finite acting groups
the entropy is equal to the Shannon entropy of the space divided by the cardinality of the group
and hence is preserved under any orbit equivalence.
The problem thus reduces to establishing
the entropy inequality $h(\alpha ) \leq h(\beta )$ when $G$ is infinite and virtually cyclic.
As in \cite{Aus16,KerLi21}, the strategy is to iteratively apply the cocycle identity in order to
bound the cocycle complexity
across sparse but coarsely dense subsets of F{\o}lner subsets of $G$, except that now these subsets must be
highly separated (e.g, in the case of $\Zb$, sparse sets of points of roughly uniform distribution
inside large intervals), so that we can no longer argue in
a geometric way entirely within $G$.
To compensate for this we employ the algebraic structure of $H$ as a finitely generated virtually Abelian group
in tandem with the ``almost linear'' structure of $G$ as a virtually cyclic group.
Unlike in \cite{KerLi21}, where no restrictions are imposed on $H$
beyond amenability or soficity, the hypotheses on $H$ of virtual Abelianness and finite generation
are now required to secure the desired small exponential complexity
(see Lemmas~\ref{L-product} and \ref{L-zero}).
Another novelty of our argument is the use of the
Jewett--Krieger theorem, which, for the purpose of obtaining lower bounds, allows us to express
dynamical entropy by means of the cardinality of certain partitions instead of their Shannon entropy.

We do not know whether the statement of Theorem~\ref{T-entropy} holds when $G$ and $H$ are
non-locally-finite countably infinite amenable groups
and one of them is assumed to be virtually Abelian and finitely generated
(although one does get an entropy inequality by \cite{KerLi21}).
In the proof of Austin's result the virtually cyclic case can be dealt with separately
since, by a result of  Bowen \cite{Aus16b}, if two
free p.m.p.\ actions of finite generated groups are integrably orbit equivalent then the groups have the
same growth, so that either both groups are virtually cyclic or neither is.
This conclusion no longer holds for Shannon orbit equivalence, as illustrated in Remark~\ref{R-odometer}.
Nevertheless, by concentrating our attention on single groups (i.e., the case $G=H$)
we can use Theorem~\ref{T-entropy} as follows to resolve the remaining case,
within the amenable context,
of Shannon orbit equivalence rigidity for Bernoulli actions.

By work of Ornstein and Weiss \cite{OrnWei87}
(generalizing Ornstein's isomorphism theorem for $G=\Zb$ \cite{Orn70}),
two Bernoulli actions of a given countably infinite amenable group $G$ are measure conjugate
if and only if they have the same entropy. As observed in \cite{KerLi21}, this combines
with the results of Austin and Kerr--Li mentioned above to show that if
$G$ is a countably infinite amenable group which is neither virtually cyclic nor locally finite
then two Bernoulli actions of $G$ are Shannon orbit equivalent if and only if they are measure conjugate.
If $G$ is countably infinite and locally finite then this conclusion fails in a rather dramatic way,
as we saw above.
In view of the Ornstein--Weiss entropy classification of Bernoulli actions,
Theorem~\ref{T-entropy} now completes the picture for countably infinite amenable groups
by covering the remaining virtually cyclic case, so that we can assert the following.

\begin{theoremi}\label{T-Bernoulli}
Two Bernoulli actions of a non-locally-finite countable amenable group are Shannon orbit equivalent
if and only if they are measure conjugate.
\end{theoremi}

By a theorem of Belinskaya \cite{Bel68}, two aperiodic ergodic p.m.p.\ transformations are integrably orbit equivalent
if and only if they are flip conjugate (meaning that they are either measure conjugate
or one is measure conjugate to the inverse of the other).
Recently Carderi, Joseph, Le Ma\^{i}tre, and Tessera showed that,
for ergodic p.m.p.\ transformations, Shannon orbit equivalence is strictly weaker than
flip conjugacy \cite{CarJosLeMTes22},
and so together these results illustrate that Theorem~\ref{T-entropy}
is not merely a formal strengthening of the analogous assertion for integrable orbit equivalence from \cite{Aus16}.
What Carderi, Joseph, Le Ma\^{i}tre, and Tessera demonstrate is that if $T$ is any aperiodic
ergodic p.m.p.\ transformation such that $T^n$ is ergodic for some $n\geq 2$ then there is another
ergodic p.m.p.\ transformation to which $T$ is Shannon orbit equivalent
(in fact orbit equivalent in a certain stronger quantitative sense)
but not flip conjugate.

In the second part of the paper we further analyze this gap between Shannon and integrable orbit equivalence
by establishing the following theorem in the odometer context.
Recall that, by a theorem of Halmos and von Neumann \cite{HalvNe42}, discrete spectrum p.m.p.\ transformations are
determined up to measure conjugacy by their eigenvalues with multiplicity,
and even up to flip conjugacy since the eigenvalues form a group.
This applies in particular to odometers, where the eigenvalues are all roots of unity.
In this case the eigenvalue information is encoded by a supernatural number
$\prod_p p^{k_p}$ where $p$ ranges over the
primes and each $k_p$ belongs to $\{ 0,1,2,\dots , \infty \}$ (see Section~\ref{S-odometer}).
The universal odometer is defined by the
condition that $k_p = \infty$ for each $p$.

\begin{theoremi}\label{T-odometer}
Every $\Zb$-odometer is Shannon orbit equivalent to the universal $\Zb$-odometer.
\end{theoremi}

Theorem~\ref{T-odometer} provides us with example of an uncountable family
of ergodic p.m.p.\ transformations such that no two of them are integrably orbit equivalent
although one of them is Shannon orbit equivalent to all of the others. We point out that it is unknown
whether Shannon orbit equivalence is a transitive relation, and so we are not claiming
that $\Zb$-odometers are all Shannon orbit equivalent to each other, which seems to be a
much more difficult problem.

Our proof of Theorem~\ref{T-odometer} should be compared to the proof of the theorem of Dye
which asserts that any two aperiodic ergodic p.m.p.\ actions on standard probability spaces are orbit equivalent
(see \cite{KatWei91} or Section~4.9 of \cite{KerLi16}). What is special in our case is that the Rokhlin towers are
canonically given to us in a nested way by the odometer structure, so that we have a certain
combinatorial rigidity that allows us to scramble information
across orbits with some uniform quantitative control.
However, we do not see how to implement the required amount of control by simply following the recursive procedure
used to establish Dye's theorem in \cite{KatWei91,KerLi16}, and so we
have developed a more intricate doubly recursive construction that will do the job.

The proofs of Theorems~\ref{T-entropy} and \ref{T-odometer} will be carried out in Sections~\ref{S-entropy}
and \ref{S-odometer}, respectively, after we set up some notation and terminology in Section~\ref{S-preliminaries}.
\medskip

\noindent{\it Acknowledgements.}
The first author was supported by the Deutsche Forschungsgemeinschaft
(DFG, German Research Foundation) under Germany's Excellence Strategy EXC 2044-390685587,
Mathematics M{\"u}nster: Dynamics--Geometry--Structure, and by the SFB 1442 of the DFG.
The second author was supported by NSF grant DMS-1900746.

\section{Preliminaries}\label{S-preliminaries}

Partitions of measure spaces are always assumed to be measurable,
and we always ignore sets of measure zero in a partition, so that when we speak of the cardinality
of a partition we mean the cardinality of the collection of its nonnull members.
If $\sP$ is a partition and $F$ is a finite subset of a group acting on the space
then we write $\sP^F$ for the join $\bigvee_{g\in F} g^{-1} \sP$.

Let $G\curvearrowright (X,\mu )$ and $H\curvearrowright (Y,\nu )$ be free
p.m.p.\ actions of countable groups.
The actions are {\it orbit equivalent} if there is a measure isomorphism $\varphi$ from a $G$-invariant
conull set $X_0 \subseteq X$ to an $H$-invariant conull set $Y_0 \subseteq Y$ such that $\varphi (Gx) = H\varphi (x)$
for all $x\in X_0$. Such a map $\varphi$ is called an {\it orbit equivalence}, and associated to it are cocycles
$\kappa : G\times X \to H$ and $\lambda : H\times Y\to G$ defined a.e.\ by
$\kappa (g,x)\varphi (x) = \varphi (gx)$ and $\lambda (s,y)\varphi^{-1} (y) = \varphi^{-1} (sy)$, with
freeness guaranteeing a.e.\ uniqueness. They satisfy the {\it cocycle identities}
$\kappa (gh,x) = \kappa (g,hx) \kappa (h,x)$ and $\lambda (st,y) = \lambda (s,ty) \lambda (t,y)$.
From the cocycle $\kappa$ we obtain, for every $g\in G$, a countable partition
$\{ X_{g,s} : s\in H \}$ of $X$ where $X_{g,s} = \{ x\in X : \kappa (g,x) = s \}$. The cocycle $\lambda$
similarly yields a partition $\{ Y_{s,g} : g\in G \}$ of $Y$ for every $s\in H$.
We refer to these partitions as the {\it cocycle partitions} associated to $\kappa$ and $\lambda$.

The {\it Shannon entropy} of a countable partition $\sP$ of a probability space $(X,\mu )$ is defined by
$H(\sP ) = \sum_{A\in\sP} -\mu (A) \log \mu (A)$, and its conditional version with respect
to a second countable partition $\sQ$
by $H(\sP |\sQ ) = \sum_{B\in\sQ} \sum_{A\in\sP} -\mu (A\cap B) \log [\mu (A\cap B)/\mu (B)]$.
We also employ the notation $H(\sP )$, using the same formula,
when $\sP$ is any countable disjoint collection of measurable subsets of $X$. In Section~\ref{S-odometer}
we use the fact that if $\sP$ is a countable partition of $X$ and $\{ B_i \}_{i\in I}$ is a countable
collection of measurable subsets of $X$ with $\bigcup_{i\in I} B_i = X$
then $H(\sP ) \leq \sum_{i\in I} H(\sP_{B_i} )$ where
$\sP_{B_i}$ denotes the restriction of $\sP$ to $B_i$, i.e., the partition $\{ A\cap B_i : A\in\sP \}$
of $B_i$. This is a consequence
of the observation that for all disjoint measurable sets $A,B\subseteq X$ one has
$H(\{ A\sqcup B \}) \leq H(\{ A,B \}) = H(\{ A \}) + H(\{ B \})$.
In Section~\ref{S-entropy} we use standard facts about Shannon entropy that can be found
for example in Sections~9.1 and 9.2 of \cite{KerLi16}.

Returning to the orbit equivalence scenario from above, if the cocycle partitions associated to $\kappa$
all have finite Shannon entropy then
we say that $\kappa$ is {\it Shannon}, and if both $\kappa$ and $\lambda$ are Shannon
then we refer to $\varphi$ as a {\it Shannon orbit equivalence}.
When such a $\varphi$ exists we say that the actions $G\curvearrowright (X,\mu )$
and $H\curvearrowright (Y,\nu )$ are {\it Shannon orbit equivalent}.
Note that, by the cocycle identities and the subadditivity of Shannon entropy with respect to joins,
the partitions $\{ X_{g,s} : s\in H \}$ for $g\in G$ all have finite Shannon entropy as soon as we know
that the ones for $g$ belonging to some generating set do.

In the case that $G$ and $H$ are finitely generated and are equipped with word-length metrics $|\!\cdot\! |_G$
and $|\!\cdot\! |_H$, we can ask whether the integrals $\int_X |\kappa (g,x)|_H \, d\mu$ and
$\int_Y |\lambda (s,x)|_G \, d\nu$ are finite for all $g\in G$ and $s\in H$, in which case we say that
$\varphi$ is an {\it integrable orbit equivalence} (this doesn't depend on the choice of word-length metrics).
The actions are {\it integrably orbit equivalent} if such
a $\varphi$ exists. A theorem of Belinskaya says that two aperiodic ergodic p.m.p.\ $\Zb$-actions
are integrably orbit equivalent if and only if they are flip conjugate, i.e., measure conjugate up to
an automorphism of $\Zb$ \cite{Bel68} (see the appendix in \cite{CarJosLeMTes22} for a short proof).
For $\Zb^2$-actions, however, even bounded
orbit equivalence (in which the cocycles have finite image when restricted to each group element) is
considerably weaker than measure conjugacy modulo a group automorphism \cite{FieFri86}.
By Lemma~2.1 of \cite{Aus16}, integrable orbit equivalence implies
Shannon orbit equivalence.

Suppose that $G$ is amenable and let $\{ F_n \}$ be a F{\o}lner sequence for $G$,
i.e., a sequence of nonempty finite subsets of $G$ satisfying
$\lim_{n\to\infty} |gF_n \Delta F_n |/|F_n| = 0$ for all $g\in G$.
The entropy of
a finite partition $\sP$ under the action $G\stackrel{\alpha}{\curvearrowright} (X,\mu )$ is defined by
\[
h(\alpha,\sP ) = \lim_{n\to\infty} \frac{1}{|F_n |} H(\sP^{F_n} ) ,
\]
a limit that always exists and
is equal to
\[
\inf_F \frac{1}{|F|} H(\sP^F )
\]
where $F$ ranges over all nonempty finite subsets of $G$ (see Section~9.3 of \cite{KerLi16}).
The entropy of the action is then defined by
\[
h(\alpha ) = \sup_\sP h(\alpha,\sP )
\]
where $\sP$ ranges over all finite partitions of $X$.
For more details see Chapter~9 of \cite{KerLi16}.

\section{Entropy and Shannon orbit equivalence}\label{S-entropy}

Given two finite subsets $F$ and $K$ of a group $G$ and a $\delta > 0$, we say that $F$ is
{\it $(K,\delta )$-invariant} if $| \{ t\in F : Kt \subseteq F \} | \geq (1-\delta ) |F|$.

\begin{lemma}\label{L-invt}
Let $G$ and $H$ be countable amenable groups. Let $G\stackrel{\alpha}{\curvearrowright} (X,\mu )$
and $H\stackrel{\beta}{\curvearrowright} (Y,\nu )$ be free p.m.p.\ actions
which are orbit equivalent, and
let $\kappa : G\times X\to H$ and $\lambda : H\times Y\to G$ be the associated cocycles.
Let $\{ F_n \}$ be a F{\o}lner sequence for $G$.
Let $L$ be a finite subset of $H$ and $\delta > 0$.
Then for every sufficiently large $n\in\Nb$ there is a set
$X_n\subseteq X$ with $\mu (X_n)\geq 1-\delta$ such that
the sets $\kappa (F_n ,x)$ for $x\in X_n$ are $(L,\delta )$-invariant.
\end{lemma}

\begin{proof}
By conjugating $\beta$ by an orbit equivalence, we may assume that both actions are on
$(X,\mu )$ and that the identity map on $X$ is an orbit equivalence between $\alpha$ and $\beta$.
Take a set $V \subseteq X$ such that $\mu (V) \geq 1-\delta^2 /2$ and $\lambda (L,V)$ is finite.
For every $n\in\Nb$ we have
\[
\int_X \frac{1}{|F_n|} \sum_{t\in F_n} 1_{\alpha_{t^{-1}} (X\setminus V)} \, d\mu
= \frac{1}{|F_n|}  \sum_{t\in F_n} \int_X 1_{\alpha_{t^{-1}} (X\setminus V)} \, d\mu
= \mu (X\setminus V) \leq \frac{\delta^2}{2}
\]
so that there exists an $X_n\subseteq X$ with $\mu (X_n) \geq 1-\delta$
such that $|F_n|^{-1} \sum_{t\in F_n} 1_{\alpha_{t^{-1}} (X\setminus V)}$
is bounded above by $\delta /2$ on $X_n$, that is,
$|\{ t\in F_n : \alpha_t x \in V \} | \geq (1-\delta /2)|F_n |$ for all $x\in X_n$.
Now suppose that $n$ is large enough so that $F_n$ is $(\lambda (L,V),\delta /2)$-invariant.
Let $x\in X_n$. For every $h\in L$ and $t\in G$ we have
\begin{align*}
\beta_{h\kappa (t,x)} x = \beta_h \alpha_t x = \alpha_{\lambda (h,\alpha_t x)} \alpha_t x
= \alpha_{\lambda (h,\alpha_t x)t} x ,
\end{align*}
and so in the case that $\alpha_t x\in V$ we get
$\beta_{L\kappa (t,x)} x\subseteq \alpha_{\lambda (L,V)t} x$. Let $F_n'$ be the set of all
$t\in F_n$ such that $\alpha_t x \in V$. Then $|F_n' | \geq (1-\delta /2)|F_n|$, and so
\begin{align*}
| \{ h\in \kappa (F_n',x) : Lh\subseteq \kappa (F_n,x) \} |
&\geq | \{ t\in F_n' : \lambda (L,V)t\subseteq F_n \} | \\
&\geq | \{ t\in F_n : \lambda (L,V)t\subseteq F_n \} | - |F_n\setminus F_n' |\\
&\geq (1-\delta /2)|F_n| - (\delta /2)|F_n| \\
&= (1-\delta )|\kappa (F_n,x)| ,
\end{align*}
that is, $\kappa (F_n,x)$ is $(L,\delta )$-invariant.
\end{proof}

\begin{lemma}\label{L-product}
Let $G$ be a finitely generated virtually Abelian group, equipped with a word metric. Let $r,n\in\Nb$.
Let $\Omega_0 \subseteq \{ 1,\dots , n \}$ and set $\Omega_1 = \{ 1,\dots ,n \} \setminus \Omega_0$.
For every $i\in \Omega_0$ let $b_i$ be a fixed element of $G$.
Then the number of elements of $G$ of the form $b_1 \cdots b_n$
where $b_i$ is an element of the $r$-ball around $e_G$ for every $i\in \Omega_1$
is at most $e^{c_1 |\Omega_0|} c_2 (rn)^k$ where $k$ is the order of polynomial growth of $G$
and $c_1 , c_2 >0$ are constants that do not depend on $r$, $n$, or the elements
$b_i$ for $i\in \Omega_0$.
\end{lemma}

\begin{proof}
By hypothesis $G$ has an Abelian subgroup $A$ of finite index, which we may assume to be normal
by replacing it with the intersection of all of its conjugates, which is again of finite index.
Since finite-index subgroups of finitely generated groups are also finitely generated, the subgroup $A$
is finitely generated. Fix a finite symmetric generating set $S$ for $A$, and
choose a set $\{ g_1 , \dots , g_l \}$ of representatives for the cosets of $A$ with $g_1 = e_G$.
Set $F = \{ g_2 , \dots , g_l , g_2^{-1} , \dots , g_l^{-1} \}$.
Since any two word metrics on $G$ are bi-Lipschitz equivalent, we may assume
that the given word metric is with respect to the symmetric generating set $S \cup F$.
Write $B(m)$ for the ball of radius $m$ around $e_G$.
Since $G$ has polynomial growth of order $k$, there exists a $C>0$ such that
$|B(m)| \leq Cm^k$ for every $m\in\Nb$.
Since $A$ has finite index in $G$, the set of all
automorphisms of $A$ of the form $a\mapsto gag^{-1}$ for some $g\in G$ is finite,
and so there is a finite set $E\subseteq G$ such that each of these automorphisms
has the form $a\mapsto gag^{-1}$ for some $g\in E$.
Set $K = \{ e_G \} \cup \bigcup_{g\in E} g(S\cup F)g^{-1}$, which is a finite subset of $G$, so that there
exists a $d\in\Nb$ for which $K\subseteq B(d)$.

Suppose we are given $b_i \in B(r)$ for $i\in \Omega_1$. For each $i=1,\dots ,n$ we can write
$b_i = a_i g_{l_i}$ for some $a_i \in A$ and $1\leq l_i \leq l$.
Using that elements of $A$ commute, and writing $h_i = g_{l_1} \cdots g_{l_{i-1}}$ for $1 < i\leq n$ and $h_1 = e_G$,
we have
\begin{align*}
b_1 \cdots b_n
&= a_1 g_{l_1} \cdots a_n g_{l_n} \\
&= \bigg( \prod_{i=1}^n h_i a_i h_i^{-1} \bigg)
g_{l_1} \cdots g_{l_n} \\
&= \bigg( \prod_{i\in\Omega_0} h_i a_i h_i^{-1} \bigg)
\bigg( \prod_{i\in\Omega_1} h_i a_i h_i^{-1} \bigg)
g_{l_1} \cdots g_{l_n} .
\end{align*}
Since $a_i$ is fixed for every $i\in\Omega_0$, there are at most $|E|^{|\Omega_0 |}$
possibilities for $\prod_{i\in\Omega_0} h_i a_i h_i^{-1}$.
Now for every $a\in B(r+1) \cap A$ and $h\in G$ the element $hah^{-1}$ is equal to $gag^{-1}$
for some $g\in E$ and thus belongs to $K^{r+1}$, which is contained in $B((r+1)d)$.
Since for every $i\in \Omega_1$ we have $a_i = b_i g_{l_i}^{-1} \in B(r+1)$, it follows that
$\prod_{i\in\Omega_1} h_i a_i h_i^{-1}\in B((r+1)d |\Omega_1 |) \subseteq B((r+1)d n)$,
and so the product $(\prod_{i\in\Omega_1} h_i a_i h_i^{-1} )g_{l_1} \cdots g_{l_n}$
lies in $B((r+1)d +1)n)$, which has cardinality at most $C((r+1)d +1)n)^k$.
Therefore the total number of possibilities for $b_1 \cdots b_n$ is bounded above by
$|E|^{|\Omega_0 |}C((r+1)d +1)n)^k$.
We can thus take $c_1 = \log |E|$ and $c_2 = C(3d)^k$.
\end{proof}

Let $G$ be an infinite virtually cyclic group and $H$ a countably infinite group.
Let $G\overset{\alpha}\curvearrowright (X, \mu)$ and $H\overset{\beta}\curvearrowright (X, \mu)$
be free ergodic p.m.p.\ actions such that the identity map on $X$ is an orbit equivalence between the
actions and the cocycle $\kappa: G\times X\rightarrow H$ is Shannon.
For each $g\in G$, denote by $\sQ_g$ the countable partition of $X$ consisting of
the sets $\{x\in X: \kappa(g, x)=h\}$ for $h\in H$. Let $a\in G$ be such that
$\left<a\right>$ is a finite-index normal subgroup of $G$.
Such an element always exists since the intersection of the conjugates of any finite-index cyclic subgroup of $G$
is a normal cyclic subgroup of finite index.
For all $n, m\in \Nb$
the join $\bigvee_{j=0}^{n-1}a^{-jm}\sQ_{a^m}$ refines $\sQ_{a^{nm}}$, and so
\[
\frac{1}{|G: \left<a^m\right>|}H(\sQ_{a^m})\ge \frac{1}{|G: \left<a^{nm}\right>|}
H\bigg(\bigvee_{j=0}^{n-1}a^{-jm}\sQ_{a^m}\bigg)\ge \frac{1}{|G: \left<a^{nm}\right>|}H(\sQ_{a^{nm}}).
\]
Since $\sQ_{a^{-m}}=a^m\sQ_{a^m}$ one also has
\[
\frac{1}{|G: \left<a^{-m}\right>|}H(\sQ_{a^{-m}})
=\frac{1}{|G: \left<a^m\right>|}H(a^m\sQ_{a^m})=\frac{1}{|G: \left<a^m\right>|}H(\sQ_{a^m}).
\]
Since any two finite-index subgroups of an infinite group have nontrivial intersection,
it follows that $\inf_{m\in \Nb}\frac{1}{|G: \left<a^m\right>|}H(\sQ_{a^m})$
is equal to the infimum of $\frac{1}{|G: \left<g\right>|}H(\sQ_g)$
over all generators $g$ of finite-index normal subgroups of $G$ and hence is a numerical invariant
of the orbit equivalence between $\alpha$ and $\beta$. We denote this quantity by $h(\kappa)$.

\begin{lemma}\label{L-zero}
If in the above setting $H$ is virtually Abelian and finitely generated, then $h(\kappa)=0$.
\end{lemma}

\begin{proof}
Take an $a\in G$ that generates a finite-index normal subgroup.
Then we have $h(\kappa)=\inf_{n\in \Nb}\frac{1}{|G:\left<a\right>|n}H(\sQ_{a^n})$.
Choose a set $R$ of coset representatives for $\left<a\right>$ with $e_G\in R$.
For each $n\in \Nb$, put $I_n=\{e_G, a, \dots, a^{n-1}\}$ and $F_n=I_nR$.
Let $\varepsilon>0$.

Take a finite subset $\sQ_a'$ of $\sQ_a$ such that setting
$C=X\setminus \bigcup \sQ_a'$ and $\sD = (\sQ_a\setminus \sQ_a')\cup \{X\setminus C\}$ one has
$\mu(C)<\varepsilon/(2|R|)$ and $H(\sD)<\varepsilon$. Set $\sC=\{C, X\setminus C\}$.

Note that each member of $\sD$ is the union of at most $|\sQ_a'|$ many members of $\sQ_a$.
Let $n\in \Nb$ and let $D$ be a nonnull member of $\sD^{I_n}$.
Then $D$ is the union of at most $|\sQ_a'|^n$ many members of $\sQ_a^{I_n}$. Thus
\begin{gather}\label{E-Qa}
\sum_{B\in \sQ_a^{I_n}, \,B\subseteq D}-\frac{\mu(B)}{\mu(D)}\log \frac{\mu(B)}{\mu(D)}\le \log |\sQ_a'|^n.
\end{gather}

Take $0<\delta<\varepsilon$ such that $\delta\log |\sQ_a'|<\varepsilon$. By the mean ergodic theorem,
when $n\in \Nb$ is sufficiently large there is a collection $\sC_n\subseteq \sC^{F_n}$ such that
$\mu(\bigcup \sC_n)>1-\delta$ and for each $A\in \sC_n$ one has
$\frac{1}{|F_n|}\sum_{g\in F_n}1_C(gx)\le \mu(C)+\varepsilon/(2|R|)$ for all $x\in A$. Denote by $\sC_n'$ the set
of all $B\in \sC^{I_n}$ containing some nonnull $A\in \sC_n$.
Then $\mu(\bigcup \sC_n')\ge \mu(\bigcup \sC_n)\ge 1-\delta$. For each $B\in \sC_n'$ and $x\in B$,
taking an $A\in \sC_n$ with $A\subseteq B$ and a $y\in A$ one has
\begin{align}\label{E-eps}
\frac{1}{|I_n|}\sum_{g\in I_n}1_C(gx)
&=\frac{1}{|I_n|}\sum_{g\in I_n}1_C(gy)\\
&\le \frac{1}{|I_n|}\sum_{g\in F_n}1_C(gy)\notag \\
&=\frac{|R|}{|F_n|}\sum_{g\in F_n}1_C(gy)\notag \\
&\le |R|\bigg(\mu(C)+\frac{\varepsilon}{2|R|}\bigg)\notag \\
&\le \varepsilon.\notag
\end{align}
Denote by $\sD_n'$ the collection of members of $\sD^{I_n}$ contained in some member of $\sC_n'$.
Then $\mu(\bigcup \sD_n')=\mu(\bigcup \sC_n')>1-\delta$, and therefore, using (\ref{E-Qa}),
\begin{align}\label{E-ne}
\sum_{D\in \sD^{I_n}\setminus \sD_n'}\,\sum_{B\in \sQ_a^{I_n},\, B\subseteq D}-\mu(D)\cdot \frac{\mu(B)}{\mu(D)}\log \frac{\mu(B)}{\mu(D)}
&\le \sum_{D\in \sD^{I_n}\setminus \sD_n'}\mu(D)\log |\sQ_a'|^n \\
&\le n\delta \log |\sQ_a'| \notag \\
&<n\varepsilon. \notag
\end{align}
Write $k$ for the order of polynomial growth of $H$ and let $r\in\Nb$ be
such that the set
$\{ \kappa (a,x) : x\in \bigcup\sQ_a' \}$ is contained
in the $r$-ball around $e_H$ with respect to some fixed word metric on $H$.
Then by Lemma~\ref{L-product} there are $c_1 , c_2 > 0$
not depending on $n$
such that every $D\in \sD_n'$ intersects at most $e^{c_1 |\Omega_D|}c_2 (rn)^k$ many members
of $\sQ_{a^n}$ where $\Omega_D=\{g\in I_n: gx\in C\}$ for $x\in D$, and for such a $D$ we have
$|\Omega_D|\le \varepsilon n$ by (\ref{E-eps}) and hence
\begin{gather}\label{E-log}
\sum_{B\in \sQ_{a^n}\vee \sD^{I_n},\, B\subseteq D}-\frac{\mu(B)}{\mu(D)}\log \frac{\mu(B)}{\mu(D)}
\le \log (e^{\varepsilon c_1 n} c_2 (rn)^k ).
\end{gather}
Denote by $\sP_n$ the partition of $X$ consisting of the members of $\sQ_a^{I_n}$ contained in
$\bigcup (\sD^{I_n}\setminus \sD_n')$ along with the members of $\sQ_{a^n}\vee \sD^{I_n}$
contained in $\bigcup \sD_n'$. Then $\sP_n$ refines both $\sQ_{a^n}$ and $\sD^{I_n}$,
and since $H(\sD^{I_n} ) \leq |I_n| H(\sD ) < n\eps$ we therefore obtain, using (\ref{E-ne}) and (\ref{E-log}),
\begin{align*}
H(\sQ_{a^n})&\le H(\sP_n)\\
&= H(\sD^{I_n})+H(\sP_n|\sD^{I_n})\\
&\le n\varepsilon+\sum_{D\in \sD^{I_n}\setminus \sD_n'}\,\sum_{B\in \sQ_a^{I_n},\, B\subseteq D}-\mu(D)\cdot \frac{\mu(B)}{\mu(D)}\log \frac{\mu(B)}{\mu(D)}\\
&\hspace*{20mm} \ +\sum_{D\in \sD_n'}\,\sum_{B\in \sQ_{a^n}\vee \sD^{I_n},\, B\subseteq D}-\mu(D)\cdot \frac{\mu(B)}{\mu(D)}\log \frac{\mu(B)}{\mu(D)}\\
&\le n\varepsilon+n\varepsilon+\sum_{D\in \sD_n'}\mu(D)\log (e^{\varepsilon c_1 n} c_2 (rn)^k )\\
&\le n\varepsilon (2+c_1 )+\log (c_2 (rn)^k).
\end{align*}
As none of $r$, $c_1$, $c_2$, and $k$ depend on $n$, this yields
\[
h(\kappa)=\inf_{n\in \Nb}\frac{1}{|G:\left<a\right>|n}H(\sQ_{a^n})\le \frac{1}{|G:\left<a\right>|}\varepsilon(2+c_1).
\]
Letting $\varepsilon\to 0$ we obtain $h(\kappa)=0$.
\end{proof}

\begin{question}
Is there any way to extend the definition of $h(\kappa)$ to other amenable groups $G$?
\end{question}

Let $G\stackrel{\alpha}{\curvearrowright} (X,\mu )$
be a free ergodic p.m.p.\ action of a countably infinite amenable group.
A finite partition $\sP$ of $X$ is said to be {\it $\alpha$-uniform} if the convergence in the
pointwise ergodic theorem applied to each of the indicator functions of members of $\sP$ is uniform off of a null set,
i.e., one obtains a uniquely ergodic subshift action of $G$ by using $\sP$ in the obvious way to define a
$G$-equivariant map into the shift over $G$ with symbol set $\sP$ and
then taking the closure of the image of some $G$-invariant conull subset of $X$.
In this topological model the unique invariant Borel probability measure defines a p.m.p.\ action that is
measure conjugate to the quotient of $\alpha$ determined by the invariant sub-$\sigma$-algebra generated by $\sP$,
and $\sP$ becomes, modulo null sets, a clopen generating partition. We can
then apply the variational principle and generator theorems to conclude that the
measure entropy $h(\alpha ,\sP )$ is equal to the topological entropy of $\sP$ as a clopen partition,
which can be expressed as $\lim_{n\to\infty} \frac{1}{|F_n |} \log |\sP^{F_n} |$
for any F{\o}lner sequence $\{ F_n \}$ for $G$.

\begin{lemma}\label{L-control}
Let $G$ be an infinite virtually cyclic group and $H$ a countably infinite amenable group.
Let $G\stackrel{\alpha}{\curvearrowright} (X,\mu )$ and $H\stackrel{\beta}{\curvearrowright} (X,\mu )$
be free ergodic p.m.p.\ actions such that the identity map on $X$
is an orbit equivalence for which the cocycle $\kappa : G\times X\to H$ is Shannon. Then
\[
h(\beta)+h(\kappa)\ge h(\alpha).
\]
\end{lemma}

\begin{proof}
Take an $a\in G$ which generates a finite-index normal subgroup of $G$.
It suffices to show that $h(\beta)+\frac{1}{|G: \left<a\right>|}H(\sQ_a)\ge h(\alpha)$.
By the Jewett--Krieger theorem \cite{Jew69,Kri72,Wei85,Ros87},
every finite partition of $X$ can be approximated arbitrarily well in the Rokhlin metric
$d(\sP , \sQ ) = H(\sP | \sQ ) + H(\sQ | \sP )$
by $\beta$-uniform finite partitions. Since the function $\sQ \mapsto h(\alpha , \sQ )$
on finite partitions is continuous with respect to this metric, it is therefore enough to show,
given a $\beta$-uniform partition $\sP$, that
$h(\beta)+\frac{1}{|G: \left<a\right>|}H(\sQ_a)\ge h(\alpha, \sP)$.

Choose a set $R$ of coset representatives for $\left<a\right>$ with $e_G\in R$.
For each $n\in \Nb$ set $I_n=\{e_G, a, \dots, a^{n-1}\}$ and $F_n=RI_n$.

Let $0<\varepsilon<1/8$. Since $\sP$ is $\beta$-uniform, so is $\sP^R$. Thus there are some nonempty finite
set $K\subseteq H$ and $\delta>0$ such that for any nonempty $(K, \delta)$-invariant finite set
$L\subseteq H$ one has $h(\beta, \sP^R)+\varepsilon \ge \frac{1}{|L|}\log |\sP^{RL}|$
(recall our convention that we do not count null sets when taking the cardinality of a partition).

For each $n\in \Nb$ we have
\[
H(\sP^{F_n}\vee \sQ_a^{I_n})\ge H(\sP^{F_n})\ge |F_n|h(\alpha, \sP)
\]
and
\[
H(\sQ_a^{I_n})\le |I_n|H(\sQ_a) ,
\]
so that
\begin{gather}\label{E-PFn1}
H(\sP^{F_n}\vee \sQ_a^{I_n}|\sQ_a^{I_n})=H(\sP^{F_n}\vee \sQ_a^{I_n})-H(\sQ_a^{I_n})\ge |F_n|h(\alpha, \sP)-|I_n|H(\sQ_a).
\end{gather}
For each nonnull $C\in \sQ_a^{I_n}$, one has
\begin{gather}\label{E-PFn2}
\sum_{B\in \sP^{F_n}\vee\sQ_a^{I_n},\, B\subseteq C}-\frac{\mu(B)}{\mu(C)}\log  \frac{\mu(B)}{\mu(C)}\le \log |\sP|^{|F_n|}=|F_n|\log |\sP|.
\end{gather}

By Lemma~\ref{L-invt},
when $n\in \Nb$ is sufficiently large there is an $X_n\subseteq X$ such that $\mu(X_n)>1-\varepsilon$ and
for each $x\in X_n$ the set $\kappa(F_n, x)$ is $(K, \delta)$-invariant. Denote by $\sC_n$ the collection of
all $C\in \sQ_a^{I_n}$ satisfying $\mu(C\cap X_n)>0$.
Then $\bigcup (\sQ_a^{I_n}\setminus \sC_n)\le \mu(X\setminus X_n)<\varepsilon$. Using (\ref{E-PFn2}) we obtain
\begin{align*}
\lefteqn{\sum_{C\in \sQ_a^{I_n}\setminus \sC_n} \mu(C)\sum_{B\in \sP^{F_n}\vee\sQ_a^{I_n},\, B\subseteq C}-\frac{\mu(B)}{\mu(C)}\log  \frac{\mu(B)}{\mu(C)}} \hspace*{35mm} \\
\hspace*{35mm} &\le \sum_{C\in \sQ_a^{I_n}\setminus \sC_n} \mu(C)|F_n|\log |\sP|\le \varepsilon |F_n|\log |\sP|
\end{align*}
and hence, combining with (\ref{E-PFn1}),
\begin{align*}
\lefteqn{\sum_{C\in \sC_n} \mu(C)\sum_{B\in \sP^{F_n}\vee\sQ_a^{I_n},\, B\subseteq C}-\frac{\mu(B)}{\mu(C)}\log  \frac{\mu(B)}{\mu(C)}} \hspace*{10mm} \\
\hspace*{10mm} &=H(\sP^{F_n}\vee \sQ_a^{I_n}|\sQ_a^{I_n})-\sum_{C\in \sQ_a^{I_n}\setminus \sC_n} \mu(C)\sum_{B\in \sP^{F_n}\vee\sQ_a^{I_n},\, B\subseteq C}-\frac{\mu(B)}{\mu(C)}\log  \frac{\mu(B)}{\mu(C)}\\
&\ge |F_n|h(\alpha, \sP)-|I_n|H(\sQ_a)-\varepsilon |F_n|\log |\sP|.
\end{align*}
It follows that there is some $C\in \sC_n$ such that
\[
\sum_{B\in \sP^{F_n}\vee\sQ_a^{I_n},\, B\subseteq C}-\frac{\mu(B)}{\mu(C)}\log  \frac{\mu(B)}{\mu(C)}\ge |F_n|h(\alpha, \sP)-|I_n|H(\sQ_a)-\varepsilon |F_n|\log |\sP|.
\]
Denoting by $\sS_n$ the set of all $A\in \sP^{F_n}$ satisfying $\mu(A\cap C)>0$, we then have
\begin{align}\label{E-Sn}
\log |\sS_n| &\ge \sum_{B\in \sP^{F_n}\vee\sQ_a^{I_n},\, B\subseteq C}-\frac{\mu(B)}{\mu(C)}\log  \frac{\mu(B)}{\mu(C)} \\
&\ge |F_n|h(\alpha, \sP)-|I_n|H(\sQ_a)-\varepsilon |F_n|\log |\sP|. \notag
\end{align}

For each $a^k\in I_n$ note that
\[
\kappa(a^k, x)=\kappa(a, a^{k-1}x)\kappa(a, a^{k-2}x)\cdots \kappa(a, x)
\]
is the same for all $x\in C$. Thus the map $g\mapsto \kappa(g, x)$ from $I_n$ to $H$ is the same for all $x\in C$.
Denote the image of this map by $L_0$. Choose an $x_0\in X_n\cap C$ and put $L=\kappa(F_n, x_0)$.
Then $L$ is $(K, \delta)$-invariant and $L\supseteq L_0$.

For each $A\in \sS_n$ one has $A\cap C=B_A\cap  C$ for a unique $B_A\in \sP^{RL_0}$. Thus
\[
|\sP^{RL}|\ge |\sP^{RL_0}|\ge |\sS_n|
\]
and hence, using (\ref{E-Sn}) and the fact that $|L| = |F_n | = |R| |I_n|$,
\begin{align*}
h(\beta)\ge h(\beta, \sP^R) &\ge \frac{1}{|L|}\log |\sP^{RL}|-\varepsilon \\
&\ge \frac{1}{|L|}\log |\sS_n|-\varepsilon \\
&\ge  h(\alpha, \sP)-\frac{1}{|R|}H(\sQ_a)-\varepsilon \log |\sP|-\varepsilon.
\end{align*}
Letting $\varepsilon\to 0$ yields $h(\beta)\ge h(\alpha, \sP)-\frac{1}{|R|}H(\sQ_a)=h(\alpha, \sP)-\frac{1}{|G: \left<a\right>|}H(\sQ_a)$.
\end{proof}

\begin{proof}[Proof of Theorem~\ref{T-entropy}]
We may assume, by conjugating $\beta$ by a Shannon orbit equivalence,
that $(Y,\nu ) = (X,\mu)$ and that the identity map on $X$ is a Shannon orbit equivalence
between the two actions. We may also assume that the actions are ergodic in view of the ergodic decomposition
and the entropy integral formula with respect to this decomposition. We may furthermore assume
that $G$ and $H$ are infinite, for the entropy of a p.m.p.\ action of a finite group
is equal to the Shannon entropy of the space (i.e., the supremum of the Shannon entropies of all
finite partitions of the space, a quantity preserved under measure isomorphism)
divided by the cardinality of the group.

If neither $G$ nor $H$ is virtually cyclic then $h(\alpha ) = h(\beta )$ by Theorem~4.1 and Proposition~3.28
of \cite{KerLi21}. If both of $G$ and $H$ are virtually cyclic then $h(\alpha ) = h(\beta )$
by Lemmas~\ref{L-zero} and \ref{L-control}.
If only one of $G$ and $H$ is virtually cyclic, say $G$ without loss of generality,
then $h(\alpha ) \geq h(\beta )$ by Theorem~4.1 and Proposition~3.28 of \cite{KerLi21}
while $h(\alpha ) \leq h(\beta )$ by Lemmas~\ref{L-zero} and \ref{L-control}, so that we again obtain
$h(\alpha ) = h(\beta )$.
\end{proof}

\section{Odometers}\label{S-odometer}

A {\it supernatural number} is a formal product of the form $\prod_p p^{k_p}$
where $p$ ranges over the primes and each $k_p$ belongs to $\{ 0,1,2,\dots , \infty \}$.
Let $q = \prod_p p^{k_p}$ be a supernatural number having infinitely many prime factors
counted with multiplicity (i.e., there are infinitely many nonzero $k_p$ or at least one $k_p$ is $\infty$).
Take a sequence $\{ n_j \}$ of natural numbers
such that the prime $p$ occurs exactly $k_p$ times among the prime factorizations
of the $n_j$ and form the inverse limit of the cyclic groups $\Zb / (n_1 \cdots n_j )\Zb$
with connecting maps that reduce mod $n_1 \cdots n_j$. The group $\Zb$ acts
by translation on the inverse limit in a continuous and uniquely ergodic way.
The resulting ergodic p.m.p.\ action we call the {\it $q$-odometer}.
One can also construct this action by taking the product $\prod_{j=1}^\infty \{ 0,1,\dots , n_j -1 \}$
equipped with the product of uniform probability measures and letting the canonical
generator of $\Zb$ act by addition by $(1,0,0,\dots )$ with carry over to the right,
with the maximal element $(n_j - 1)_{j=1}^\infty$ being sent to the minimal one $(0,0,0,\dots )$
(from the ergodic-theoretic viewpoint this special orbit can be ignored).
In either case the specific choice of the numbers $n_j$ does not matter up to either topological or measure conjugacy,
and so there is no ambiguity in the terminology.
In general we refer to these p.m.p.\ actions as {\it odometers}, or {\it $\Zb$-odometers} if we wish to explicitly
distinguish them from actions of other residually finite groups that can be similarly constructed as inverse limits of
finite quotients.

In the case that $k_p = \infty$ for every $p$ we speak of the {\it universal odometer}.

Odometers have discrete spectrum and hence, by the Halmos--von Neumann theorem,
are classified up to measure conjugacy
by their eigenvalues with multiplicity. For the above $q$-odometer these eigenvalues
(besides $1$)
are obtained by collecting together over all primes $p$ the $p$th roots of unity counted with multiplicity $k_p$.
In particular, supernatural numbers with infinitely many prime factors (counting multiplicity) form a complete invariant for
odometers up to measure conjugacy (and even flip conjugacy since the eigenvalues form a group).

In preparation for the proof of Theorem~\ref{T-odometer} we set up some notation and terminology.
Let $(X,\mu )$ be a probability space and $T:X\to X$ an aperiodic measure-preserving transformation.
Let $S$ be a partial transformation of $X$ (i.e., a bimeasurable bijection from one measurable subset of $X$
to another) such that for every $x$ in its domain the image $Sx$ is contained in the $T$-orbit of $x$. Set
\[
\sP_{S,T} = \{ \{ x\in\dom (S) : Sx = T^n x \} : n\in\Zb \} ,
\]
which is a partition of the domain of $S$. We will make use of the Shannon entropy $H(\sP_{S,T})$
of such disjoint collections. Note that if two partial transformations $S_1$ and $S_2$
agree on the intersection of their domains then their common extension $S$
satisfies $H(\sP_{S,T} ) \leq H(\sP_{S_1,T} ) + H(\sP_{S_2,T} )$.

By a {\it ladder} for $T$ we mean a pair $(\{ C_i \}_{i=0}^{n-1} , S)$ where $C_1 , \dots , C_{n-1}$
are pairwise disjoint measurable subsets of $X$ (the {\it rungs} of the ladder) and $S$ is a measure isomorphism
from $C_0 \sqcup \dots \sqcup C_{n-2}$ to $C_1 \sqcup \dots \sqcup C_{n-1}$
such that $SC_i = C_{i+1}$ for every $i=0,\dots , n-2$ and
for every $x\in C_0 \sqcup \dots \sqcup C_{n-2}$ the point $Sx$ is contained in the $T$-orbit of $x$.
By a {\it tower} for $T$ we mean a pair $(B,n)$ where $B$ is a measurable subset of $X$, $n\in\Nb$,
and the sets $B , T^{-1} B , T^{-2} B ,\dots , T^{-(n-1)} B$ (the {\it levels} of the tower) are pairwise disjoint.
A tower is in effect a special kind of ladder, but the terminological distinction will be useful
in the proof below.

We now proceed to establish Theorem~\ref{T-odometer}, which says that every $\Zb$-odometer
is Shannon orbit equivalent to the universal $\Zb$-odometer.

\begin{proof}[Proof of Theorem~\ref{T-odometer}]
Let $q$ be a supernatural number with infinitely many prime factors counting multiplicity
and let us show that the $q$-odometer and the universal odometer are Shannon orbit equivalent.
Let $T$ be the $q$-odometer acting on the space $(X,\mu )$, for which we will give a concrete realization below.
Take a sequence $\{ p_n \}$ of prime numbers in which every prime appears infinitely often.
Take a sequence $a_1 , a_2 , \dots$ of integers greater than $1$, to be further specified.
Take another sequence $1 = d_0 < d_1 < d_2 \dots$ of integers, to be further specified,
such that the successive quotients $d_n / d_{n-1}$ for $n\geq 1$ are integers greater than $2$
and, counting multiplicity, the prime factors appearing among these quotients are the same
as those appearing in $q$.
For $n\geq 1$ set $w_n = \prod_{i=1}^n a_i$
and $v_n = \sum_{i=1}^n w_i$.
The numbers $a_n$ and $d_n$ will play a role
in the recursive construction that we carry out below and will be specified once we complete
the description of the construction. For each $n\geq 1$ the ratio $w_n / d_n$ will be
between $\frac12$ and $1$ and will
tend to $1$ as $n\to\infty$. The construction will produce a transformation
$S$ of $X$ which is measure conjugate to the universal odometer.

We may regard $X$ as the product $\prod_{n=1}^\infty \{ 0,1,\dots , \frac{d_n}{d_{n-1}} -1 \}$
and $\mu$ as the product of uniform probability measures,
with $T$ acting by addition by $(1,0,0,\dots )$ with carry over to the right.
Then for each $m\in\Nb$ we have a canonical tower decomposition
$X = \bigsqcup_{k=0}^{d_m - 1} T^{-k} B_m$
where the base $B_m$ is the set of all $(x_n )_n \in \prod_{n=1}^\infty \{ 0,1,\dots , \frac{d_n}{d_{n-1}} - 1 \}$
such that $x_n = 0$ for $n=1,\dots, m$. We will refer to this tower as the $B_m$ tower.

By a double recursion over $n\in\Nb$ (on the outside)
and $m\geq n$ (on the inside, for a fixed $n$), we will construct an array of ladders
$\sL_{n,m} = ( \{ C_{n,m,i} \}_{i=0}^{a_n-1} , S_{n,m} )$ for $T$. For a fixed $n$
the ladders $\sL_{n,m}$ for $m\geq n$ will be pairwise disjoint.
The construction will be such that
\begin{enumerate}
\item for all $n\leq l \leq m$ and $0\leq i\leq a_n-1$,
each level in the $B_m$ tower is either contained in $C_{n,l,i}$ or disjoint from it,

\item for all $m\geq n>1$ one has
$\bigsqcup_{i=0}^{a_n - 1} C_{n,m,i} \subseteq \bigsqcup_{l=n-1}^m C_{n-1,l,0}$
(i.e., the ladder $\sL_{n,m}$ is contained in the union of the bases of the ladders
$\sL_{n-1,l}$ for $l=n-1,\dots , m$).
\end{enumerate}

Now let $m \geq n \geq 1$ and suppose that we have completed the stages of the outer recursion from $1$ to $n-1$
(unless we are in the base case $n=1$) and of the inner recursion from $n$ to $m-1$
(unless we are in the base case $m=n$). We break into three cases.

(I) Case $m=n=1$. Here we define
\begin{itemize}
\item $C_{1,1,i} = T^{-i} B_1$ for $i=0,\dots , a_1 -1$,

\item $S_{1,1} = T^{-1}$ on $\bigsqcup_{i=0}^{a_1 - 2} C_{1,1,i}$.
\end{itemize}
The ladder $\sL_{1,1}$ is then defined to be $( \{ C_{1,1,i} \}_{i=0}^{a_1-1} , S_{1,1} )$.

(II) Case $m=n>1$.
By the recursive hypothesis (i), for $l=n-1,n$
each level of the $B_n$ tower is either contained in $C_{n-1,l,0}$ or disjoint from it.
Let $r_{n,n} \in\Nb$ and $0\leq s(n,n,0) < s(n,n,1) < \dots < s(n,n,r_{n,n} ) < d_n$
be such that the levels of the $B_n$ tower
contained in the set
\[
W_{n,n} := C_{n-1,n-1,0} \sqcup C_{n-1,n,0}
\]
are precisely $T^{-s(n,n,0)} B_n , T^{-s(n,n,1)} B_n , \dots , T^{-s(n,n,r_{n,n}-1)} B_n$
(the duplication of $n$ in the notation is for consistency with case III below).
We define
\begin{itemize}
\item $C_{n,n,i} = T^{-s(n,n,i)} B_n$ for every $i=0,\dots , a_n-1$,

\item $S_{n,n} = T^{-s(n,n,i+1) + s(n,n,i)}$ on $T^{-s(n,n,i)} B_n$ for every $i=0,\dots , a_n -2$.
\end{itemize}
The ladder $\sL_{n,n}$ is then defined to be $(\{ C_{n,n,i} \}_{i=0}^{a_n-1} , S_{n,n} )$.

(III) Case $m>n$.
If $n=1$ we write $W_{n,m}$ for the set
$X\setminus \bigsqcup_{l=1}^{m-1} \bigsqcup_{i=0}^{a_1-1} C_{1,l,i}$, while if $n>1$
we write $W_{n,m}$ for the set
\[
\bigg(\bigsqcup_{l=n-1}^m C_{n-1,l,0} \bigg) \setminus
\bigg(\bigsqcup_{l=n}^{m-1} \bigsqcup_{i=0}^{a_n - 1} C_{n,l,i} \bigg) .
\]
By the recursive hypothesis (i), each level of the $B_m$ tower is either contained in $W_{n,m}$
or disjoint from it.
Let $0\leq s(n,m,0) < s(n,m,1) < \dots < s(n,m,r_{n,m}) < d_m$ be such that the levels of the $B_m$ tower
contained in $W_{n,m}$ are precisely
$T^{-s(n,m,0)} B_m , T^{-s(n,m,1)} B_m , \dots , T^{-s(n,m,r_{n,m})} B_m$.
Set $t_{n,m}$ to be the largest positive integer such that
\begin{gather}\label{E-q}
a_n t_{n,m} \leq r_{n,m} .
\end{gather}
For each $j=0,\dots , t_{n,m} -1$ define
\begin{itemize}
\item $C_{n,m,i}^{(j)} = T^{-s(n,m,a_n j+i)} B_m$ for every $i=0,\dots , a_n-1$,

\item $S_{n,m}^{(j)} = T^{-s(n,m,a_n j+i+1) + s(n,m,a_n j+i)}$ on $T^{-s(n,m,a_n j+i)} B_m$
for every $i=0,\dots , a_n -2$.
\end{itemize}
Then $\sL_{n,m}^{(j)} := ( \{ C_{n,m,i}^{(j)} \}_{i=0}^{a_n-1} , S_{n,m}^{(j)} )$ for $j=0,\dots , t_{n,m} -1$ are
ladders for $T$, and they are pairwise disjoint.
We combine them to create a single ladder
$\sL_{n,m} := (\{ C_{n,m,i} \}_{i=0}^{a_n-1} , S_{n,m} )$
by setting $C_{n,m,i} = \bigsqcup_{j=0}^{t_{n,m}-1} C_{n,m,i}^{(j)}$ and
defining $S_{n,m}$ to coincide with $S_{n,m}^{(j)}$ on $\bigsqcup_{i=0}^{a_n - 2} C_{n,m,i}^{(j)}$
for every $j=0,\dots , t_{n,m} - 1$.
This completes the construction. Note that (i) and (ii) are satisfied for $n$ and $m$.

For convenience we extend the notation $t_{n,m}$ from case (III) by setting $t_{n,n} = 1$ for every $n\geq 1$.

For each $j = 0,\dots , t_{n,m} - 1$ we define the {\it spread} of the ladder $\sL_{n,m}^{(j)}$ to be the
maximum of the positive integers
$s(n,m,a_n j + i+1) - s(n,m,a_n j + i)$ over all $i=0,\dots , a_n -2$
(this is the maximum distance, in terms of levels in the $B_m$ tower, between the rungs of the ladder).
By construction one sees that when $m > n$ the spreads of the ladders
$\sL_{n,m}^{(j)}$ for $j=0,\dots , , t_{n,m} -1$
are bounded above by $d_{m-1}$ (for this it is important that there is always
at least one unused level left over at the top of the tower when building the ladders
$\sL_{n,m}$ so that the spread of the ladders at the next stage $m+1$ for a fixed $n$
is at most $d_m$, and this is guaranteed by $(\ref{E-q})$).

Let $n\in\Nb$. For each $i=0,\dots, a_n - 1$ set $C_{n,i} = \bigsqcup_{m=n}^\infty C_{n,m,i}$ and define
$S_n$ on $\bigsqcup_{m=n}^\infty \dom (S_{n,m} )$ by setting $S_n = S_{n,m}$ on $\dom (S_{n,m} )$
for every $m\geq n$. Then $\sL_n := (\{ C_{n,i} \}_{i=0}^{a_n -1} , S_n )$ is a ladder for $T$.

We use the ladders $\sL_{n,m}$ to build the transformation $S$ as follows.
For $m\geq n\geq 1$ set
\[
D_{n,m} = S_1^{a_1 - 1} S_2^{a_2 - 1} \cdots S_{n-1}^{a_{n-1} - 1}
\bigg(\bigsqcup_{i=0}^{a_n-2} C_{n,m,i} \bigg),
\]
The sets $D_{n,m}$ are pairwise disjoint and their union has measure $1$.
We then define $S$ so that on $D_{n,m}$ it is given by
\[
S_{n,m} S_{n-1}^{-a_{n-1} + 1} \cdots S_2^{-a_2 + 1} S_1^{-a_1 + 1} .
\]
Up to measure conjugacy, this yields an odometer:
for each $n$, the union of the sets $D_{n,m}$ over $m\geq n$ represents the set of points
whose coordinates in the odometer from $1$ to $n-1$ are maximum
but whose coordinate at $n$ is not maximum, so that $S$ produces a roll-over of the
first $n-1$ coordinates, increments the $n$th coordinate by $1$, and does not change any other coordinates.

We assume $d_n$ to be taken large enough to guarantee that $r_{n,n} \geq p_n$
(or $d_1 \geq p_1$ in the case $n=1$).
Then the largest multiple of $p_n$ no greater than $r_{n,n}$ (or $d_1 - 1$ in the case $n=1$)
is nonzero, and
we declare $a_n$ to be this multiple (we still need to further specify $d_n$ below, but the way in which
$a_n$ depends on $d_n$ is not affected by this).
This will ensure that
$S$ is the universal odometer. In the course of what follows we will specify how large the numbers
$d_n$ should be chosen so as to obtain a Shannon orbit equivalence.

For $m\geq n\geq 1$ define $R_{n,m}$ to be the restriction of $S_n$ to
$\bigsqcup_{l=n}^m \dom (S_{n,l} )$.
For $n\in\Nb$ and $A\subseteq \bigsqcup_{l=n}^m C_{n,l,0}$ write $\bar{R}_{n,m} A$ for the set
$\bigsqcup_{i=0}^{a_n-1} R_{n,m}^i A$ (i.e., the saturation of $A$ within the ladders
$\sL_{n,l}$ for $l=n,\dots , m$),
and for $m\geq n$ set $A_{n,m} = \bigsqcup_{l=n}^m C_{n,l,0}$ and
\[
E_{n,m} = \bar{R}_{1,m} \bar{R}_{2,m} \cdots \bar{R}_{n,m} A_{n,m}
= \bigsqcup_{i=0}^{a_1 \cdots a_n -1} S^i A_{n,m} .
\]
Note that $E_{n,m} \subseteq E_{n,m'}$ when $m' > m$, $E_{n,m} \supseteq E_{n',m}$ when $n' > n$,
and $E_{n,n} \subseteq E_{n-1,n}$ when $n > 1$.
Also, for each $n\geq 1$ the increasing union $\bigcup_{m=n}^\infty E_{n,m}$ has measure $1$.

By choosing the numbers $d_n$ to be large enough in succession,
we can arrange for the following additional conditions to hold.
First, if for $n\geq 1$ we set
\begin{gather}\label{E-beta}
\beta_n = \frac{1 + 2(v_{n-1} + p_n w_{n-1})}{d_n}
\end{gather}
then we may assume that $\lim_{n\to\infty} \beta_n = 0$ and
\begin{align}\label{E-finitesum}
\sum_{n=1}^\infty \beta_n \log \frac{9w_n^2 d_n}{\beta_n} < \infty .
\end{align}
Write $\theta_{n,m} = (1+w_{n-1} ) d_{m-1}$.
For $n\geq 3$ the ratio $w_{n-1} / d_{n-1}$ is at least
$1 - (v_{n-2} + p_{n-1} w_{n-2} )/d_{n-1}$ (see (\ref{E-Kn})) and thus can be assumed to be
no smaller than $\frac12$, so that for $n \geq 2$ we can make the quantity
$\frac{1}{w_{n-1}} \log d_{n-1}$ small enough to ensure that
\begin{gather}\label{E-n}
-\frac{1}{w_{n-1}} \log \bigg(\frac{1}{w_{n-1}} \cdot \frac{1}{\theta_{n,n}} \bigg) < \frac{1}{2^{n+2}} .
\end{gather}
We may moreover assume, for $n\geq 2$, that
\begin{gather}\label{E-n2}
-\frac{v_{n-1} + p_n w_{n-1}}{d_n} \log \bigg( \frac{v_{n-1} + p_n w_{n-1}}{d_n}
\cdot \frac{1}{\theta_{n,n+1}} \bigg) < \frac{1}{2^{n+2}} ,
\end{gather}
and also, for every $m \geq 4$ and $p=1,\dots , m-2$, that
\begin{gather*}
 -\frac{v_p}{d_{m-1}} \log \bigg( \frac{v_p}{d_{m-1}} \cdot
\frac{1}{\theta_{p,m}} \bigg) < \frac{1}{2^{m+1}} ,
\end{gather*}
which implies that for all $n\geq 2$ we have
\begin{gather}\label{E-msum}
\sum_{m=n+2}^\infty -\frac{v_n}{d_{m-1}} \log \bigg( \frac{v_n}{d_{m-1}} \cdot
\frac{1}{\theta_{n,m}} \bigg) \leq \sum_{m=n+2}^\infty \frac{1}{2^{m+1}} \leq \frac{1}{2^{n+1}} .
\end{gather}
Finally, we may also assume that
\begin{gather}\label{E-n1}
\sum_{m=2}^\infty -\frac{a_1}{d_m}
\log \frac{a_1}{d_m^2}
< \infty .
\end{gather}

Let $n\geq 1$. Define $K_n$ to be the union of the tower levels $T^{-i} B_n$ with
$0 < i \leq d_n - 1$ such that both $T^{-i}B_n$ and $T^{-(i-1)}B_n$ are contained in $E_{n, n}$,
i.e., the set of all $x\in E_{n,n} \setminus B_n$ such that $Tx \in E_{n,n}$. Given an $x\in K_n$,
we wish to show that $Tx$ can be expressed as $S^k x$ for some $k$ within certain bounds.
Since both $x$ and $Tx$ belong to $E_{n,n}$,
by construction there exist a $y\in X$ in some rung of the ladder $\sL_{n,n}$
and a $k_0 \geq 0$ with $k_0 \leq w_{n-1} -1$ such that $x = S^{k_0} y$,
as well as a $z\in X$ in some rung of the ladder $\sL_{n,n}$
and a $k_1 \geq 0$ with $k_1 \leq w_{n-1} -1$ such that $Tx = S^{k_1} z$.
Since $x=S_1^{l_1}S_2^{l_2} \cdots S_{n-1}^{l_{n-1}}y$ for some $l_1,\dots , l_{n-1}$ satisfying $0\le l_j\le a_j-1$ for all $j$,
and for all $1\leq p\leq l$ the spreads of the ladders $\sL_{p,l}^{(j)}$ for $j=0,\dots , t_{p,l} - 1$
are bounded above by $d_{l-1}$,
the jump in levels within the $B_n$ tower in going from $y$ to $x$
is at most $(a_1 + \cdots + a_{n-1} )d_{n-1}$. Similarly, the jump in levels within the $B_n$ tower in going from
$z$ to $Tx$ is at most $(a_1 + \cdots + a_{n-1} ) d_{n-1}$. Thus the jump in levels within the $B_n$ tower in going from
$y$ to $z$ is at most $2w_{n-1} d_{n-1} + 1$. Since $y$ and $z$ both lie in rungs of the ladder $\sL_{n,n}$,
it follows that we have $y = S_n^j z$ for some $j$ satisfying $|j| \leq 2w_{n-1} d_{n-1} + 1$
(this will typically be a very crude bound, governed by the extreme scenario in which
each level of the $B_n$ tower between those containing $y$ and $z$
is a rung of the ladder $\sL_{n,n}^{(0)}$), which implies
that $y = S^{k_2} z$ for some $k_2$ satisfying
\[
|k_2 | = w_{n-1} |j| \leq 3w_{n-1}^2 d_{n-1} .
\]
Putting things together, we conclude that $Tx = S^k x$ where $k$ satisfies
\begin{gather}\label{E-crude}
|k| \leq k_0 + k_1 + |k_2| \leq 2(w_{n-1} - 1) + 3w_{n-1}^2 d_{n-1} \leq 4w_{n-1}^2 d_{n-1} .
\end{gather}

Observe next that, for $n\geq 2$,
\begin{align}\label{E-Kn}
\mu (X\setminus E_{n,n} ) &\leq \sum_{j=1}^{n-1} \frac{(r_{j,n} - a_j t_{j,n} + 1)w_{j-1}}{d_n}
+ \frac{(r_{n,n} - a_n + 1)w_{n-1}}{d_n} \\
&\leq \sum_{j=1}^{n-1} \frac{a_j w_{j-1}}{d_n} + \frac{p_n w_{n-1}}{d_n} \notag \\
&= \frac{v_{n-1} + p_n w_{n-1}}{d_n} . \notag
\end{align}
Since $K_n$ can be obtained from $E_{n,n} \setminus B_n$ by removing $T^{-1} (X\setminus E_{n,n} )$,
we thereby obtain, recalling the definition of $\beta_n$ from (\ref{E-beta}),
\begin{align}\label{E-Knbeta}
\mu (X\setminus K_n )
\leq \mu (B_n ) + 2\mu (X\setminus E_{n,n} )
\stackrel{(\ref{E-Kn})}{\leq}
\beta_n .
\end{align}
Since $\beta_n \to 0$ this shows that $\mu (K_n ) \to 1$.
Thus $S$ generates the same equivalence relation as $T$ modulo a null set.
For $n>1$ we have, setting $K_n' = K_n \setminus K_{n-1}$ and using (\ref{E-Knbeta}),
\begin{gather}\label{E-Fn}
\mu (K_n' ) \leq \mu (X\setminus K_{n-1} ) \leq \beta_{n-1} .
\end{gather}
The cocycle partition $\sP_{T,S}$ of $T$ with respect to $S$ then satisfies,
using the fact that a uniform partition of a measurable set maximizes the entropy among all
partitions of the set with a given cardinality, and writing
$\lambda_{n-1} = 2\cdot 4w_{n-1}^2 d_{n-1} + 1$ for brevity,
\begin{align*}
H(\sP_{T,S} )
&\stackrel{(\ref{E-crude})}{\leq} -\mu (K_1 ) \log \mu (K_1 )
+ \sum_{n=2}^\infty \lambda_{n-1}
\bigg( -\frac{\mu (K_n' )}{\lambda_{n-1}} \log \frac{\mu (K_n' )}{\lambda_{n-1}}\bigg) \\
&\stackrel{(\ref{E-Fn})}{\leq} -\mu (K_1 ) \log \mu (K_1 ) +
\sum_{n=2}^\infty \beta_{n-1} \log \frac{9w_{n-1}^2 d_{n-1}}{\beta_{n-1}} \\
&\stackrel{(\ref{E-finitesum})}{<} \infty .
\end{align*}

Finally, we show that $H(\sP_{S,T} )$ is also finite. For $n\geq 1$ set $D_n = \bigsqcup_{m=n}^\infty D_{n,m}$.
Notice that for $m-1\geq n\geq 1$ we have $D_{n,m} \subseteq X\setminus E_{n,m-1}$, and so
in the case $m-1 > n \geq 1$ we obtain
\begin{gather}\label{E-Dnm}
\mu (D_{n,m})
\leq \mu (X\setminus E_{n,m-1} )
\leq \frac{a_1}{d_{m-1}} + \frac{a_1 a_2}{d_{m-1}} + \cdots + \frac{a_1 \cdots a_n}{d_{m-1}} = \frac{v_n}{d_{m-1}} .
\end{gather}
Suppose $m > n=1$. Then for every $x\in D_{1,m}$ we have
$Sx = T^{-k} x$ for some $1 \leq k \leq d_{m-1}$ (since the distance between successive rungs
in each $\sL_{1,m}^{(j)}$ is at most $d_{m-1}$).
Using as before the fact that a uniform partition of a measurable set
maximizes the entropy among all partitions of the set with a given cardinality, and also using
the fact that $\sP_{S|_{D_{1,1}},T}$ and $\sP_{S|_{D_{1,2}},T}$ are finite collections,
we obtain
\begin{align*}
H(\sP_{S|_{D_1} ,T} )
&\leq \sum_{m=1}^\infty H(\sP_{S|_{D_{1,m}} ,T} ) \\
&\leq H(\sP_{S|_{D_{1,1}} ,T} ) + H(\sP_{S|_{D_{1,2}} ,T} ) + \sum_{m=3}^\infty
d_{m-1} \bigg( -\frac{\mu (D_{1,m})}{d_{m-1}} \log \frac{\mu (D_{1,m})}{d_{m-1}} \bigg) \\
&\stackrel{(\ref{E-Dnm})}{\leq} H(\sP_{S|_{D_{1,1}} ,T} ) + H(\sP_{S|_{D_{1,2}} ,T} )
+ \sum_{m=3}^\infty -\frac{a_1}{d_{m-1}}
\log \bigg( \frac{a_1}{d_{m-1}}\cdot\frac{1}{d_{m-1}} \bigg) \\
&\stackrel{(\ref{E-n1})}{<} \infty .
\end{align*}

Now suppose $m\geq n>1$. By definition $S = S_{n,m} S_{n-1}^{-a_{n-1} + 1} \cdots S_1^{-a_1 + 1}$ on $D_{n,m}$.
On $\dom (S_{n,m} ) = \bigsqcup_{i=0}^{a_{n-1} - 2} C_{n,m,i}$ we have $S_{n,m} x = T^{-k} x$ for
some $0\leq k \leq d_{m-1}$.
On the other hand, for each $1\leq p \leq n-1$, on $\dom (S_{p,m} )$ we have
$S_p^{-a_p + 1} = T^k$ for some $0 \leq k \leq a_p d_{m-1}$
(using the bound $d_{m-1}$ on the spreads of the ladders $\sL_{p,m}^{(j)}$).
Therefore on $D_{n,m}$ we have $S = T^k$ for some nonzero $k$ with $-d_{m-1} \leq k \leq w_{n-1} d_{m-1}$,
and there are at most $(1+w_{n-1} ) d_{m-1}$ possibilities for this $k$.
Again using the fact that a uniform partition of a measurable set maximizes the entropy
among all partitions of the set with a given cardinality,
and writing $\theta_{n,m} = (1+w_{n-1} ) d_{m-1}$ as before, it follows that
\begin{align*}
H(\sP_{S|_{D_n} ,T} )
&\leq \sum_{m=n}^\infty H(\sP_{S|_{D_{n,m}} ,T} ) \\
&\leq \theta_{n,n} \bigg( -\frac{\mu (D_{n,n} )}{\theta_{n,n}} \log \frac{\mu (D_{n,n} )}{\theta_{n,n}} \bigg) \\
&\hspace*{20mm} \ + \theta_{n,n+1} \bigg( -\frac{\mu (D_{n,n+1} )}{\theta_{n,n+1}} \log \frac{\mu (D_{n,n+1} )}{\theta_{n,n+1}} \bigg) \\
&\hspace*{20mm} \ + \sum_{m=n+2}^\infty \theta_{n,m}
\bigg( -\frac{\mu (D_{n,m} )}{\theta_{n,m}} \log \frac{\mu (D_{n,m} )}{\theta_{n,m}} \bigg) \\
&\stackrel{(\ref{E-Dnm})}{\leq}  -\frac{1}{w_{n-1}} \log \bigg(\frac{1}{w_{n-1}} \cdot \frac{1}{\theta_{n,n}} \bigg) \\
&\hspace*{20mm} \ - \frac{v_{n-1} + p_n w_{n-1}}{d_n} \log \bigg( \frac{v_{n-1} + p_n w_{n-1}}{d_n}
\cdot \frac{1}{\theta_{n,n+1}} \bigg) \\
&\hspace*{20mm} \ + \sum_{m=n+2}^\infty -\frac{v_n}{d_{m-1}} \log \bigg( \frac{v_n}{d_{m-1}} \cdot
\frac{1}{\theta_{n,m}} \bigg) \\
&\stackrel{(\ref{E-n},\ref{E-n2},\ref{E-msum})}{\leq}
\frac{1}{2^{n+2}} + \frac{1}{2^{n+2}} + \frac{1}{2^{n+1}} = \frac{1}{2^n} .
\end{align*}
We then have
\begin{gather*}
H(\sP_{S,T} ) \leq \sum_{n=1}^\infty H(\sP_{S|_{D_n} ,T} )
< H(\sP_{S|_{D_1} ,T} ) + \sum_{n=2}^\infty \frac{1}{2^n}
< \infty . \qedhere
\end{gather*}
\end{proof}

\begin{remark}\label{R-odometer}
If one is permitted to use two different acting groups then one can show much more easily that
certain actions which are known not to be
integrably orbit equivalent are in fact Shannon orbit equivalent. This phenomenon was
observed in \cite{DelKoiLeMTes20} in the context of measure equivalence for groups.
Consider for example the odometer $\Zb$-action on $\{ 0,1,2,3 \}^\Nb$ and the action of $\Zb^2$ on
$\{ 0,1 \}^\Nb \times \{ 0,1 \}^\Nb = (\{ 0,1 \} \times \{ 0,1 \} )^\Nb$
implemented on the canonical generators by $T\times\id$ and $\id\times T$
where $T$ denotes the odometer transformation of $\{ 0,1\}^\Nb$. Let
$\sigma : \{ 0,1 \} \times \{ 0,1 \} \to \{0,1,2,3 \}$ be the bijection given by
$0\mapsto (0,0)$, $1\mapsto (0,1)$, $2\mapsto (1,0)$, and $3\mapsto (1,1)$,
and define the homeomorphism
$\Phi : (\{ 0,1 \} \times \{ 0,1 \} )^\Nb \to \{ 0,1,2,3 \}^\Nb$ by $(a_n , b_n )_n \mapsto (\sigma (a_n ,b_n ))_n$.
Then $\Phi$ is an orbit equivalence, and the cocycle partition $\sP$ associated to the generator of $\Zb$
is a coarsening of the partition $\sQ$ of $\{ 0,1,2,3 \}^\Nb$ consisting of sets of the form $\prod_n A_n$
where for some $N$ one has $A_n = \{ 3 \}$ for $n=1,\dots , N$, $A_{N+1} = \{ j \}$ for some $j\in\{ 0,1,2 \}$,
and $A_n = \{ 0,1,2,3 \}$ for $n\geq N+2$. The Shannon entropy of $\sQ$ is
$\sum_{n=1}^\infty -3\cdot 4^{-n} \log 4^{-n} < \infty$,
so that $\sP$ has finite Shannon entropy. The cocycle partitions associated to the
canonical generators of $\Zb^2$ can similarly be seen to have finite Shannon entropy, and so the
two actions are Shannon orbit equivalent.
On the other hand, a result of Bowen \cite{Aus16b} shows that if two
free p.m.p.\ actions of finite generated groups are integrably orbit equivalent then the groups must have the
same growth, which is not the case for $\Zb$ and $\Zb^2$.
\end{remark}

\end{document}